\documentclass[12pt]{article}
\usepackage{amsmath,amssymb,latexsym,amsthm}
\usepackage{url}
\usepackage{graphicx}
\newtheorem{lemma}{Lemma}

\newtheorem{theorem}{Theorem}
\newtheorem{definition}{Definition}
\newtheorem{example}{Example}

\def\R{\mathbb R}

\def\N{\mathbb N}

\def\p{\partial}

\def\sup{\mathop{\rm sup}}
\def\supp{\mathop{\rm supp}}

\def\argmin{\mathop{\rm arg min}}
\def\ll{\left}
\def\rr{\right}

\newcommand{\norm}[1]{\left\|#1 \right\|}
\def\bar{\overline}
\begin{document}
\title{Solving an inverse problem for the wave equation by using a minimization algorithm and time-reversed measurements}
\author{Lauri Oksanen, 
University of Helsinki 
}


\maketitle
\let\thefootnote\relax\footnotetext{MSC classes: 35R30}


\vspace{1cm}
{\em Abstract.} 
We consider the inverse problem for the wave equation on a compact Riemannian manifold or on a bounded domain of $\R^n$,
and generalize the concept of  
{\em domain of influence}. 
We present an efficient minimization algorithm to compute the volume of a domain of influence 
using boundary measurements and time-reversed boundary measurements. Moreover, we show that if the manifold is simple, then the volumes of the domains of influence 
determine the manifold. 
For a continuous real valued function $\tau$ on the boundary of the manifold,
the domain of influence is 
the set of those points on the manifold from which 
the travel time to some boundary point 
$y$ is less than $\tau(y)$.

\begin{figure}[h]
\centering
\includegraphics[width=0.5\textwidth,clip]{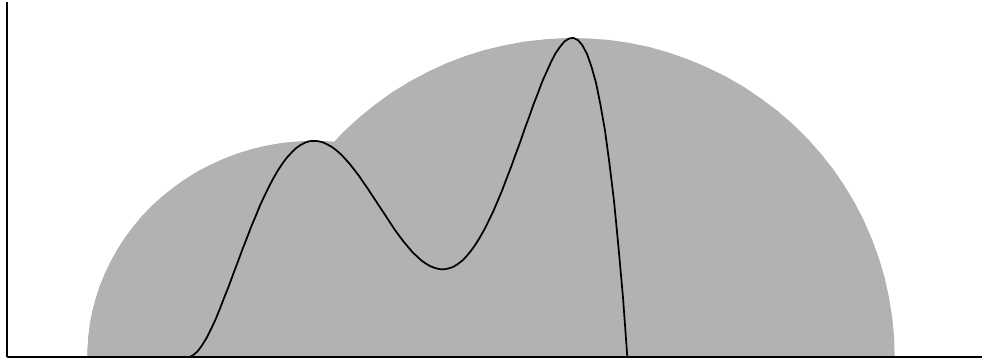}
\caption{The grey area depicts the domain of influence
and the black curve $\tau$. 
The horizontal axis is the boundary of the manifold,
and the vertical axis represents the direction of the inward pointing normal vector.
}
\end{figure}

\section{Introduction}

Let $M \subset \R^n$ be an open, bounded and connected set with a smooth boundary, and consider the wave equation
on $M$,
\begin{align}
\label{eq:wave_isotropic}
&\p_t^2 u(t,x) - c(x)^2 \Delta u(t,x) = 0, 
\quad &(t,x) \in (0, \infty) \times M,
\\&u(0,x) = 0,\ \p_t u(0,x) = 0,
\quad &x \in M, \nonumber
\\&\p_\nu u(t,x) = f(t,x), 
\quad &(t,x) \in (0, \infty) \times \p M, \nonumber
\end{align}
where $c$ is a smooth stricty positive function on $\bar M$, and $\p_\nu$ is the normal derivative on 
the boundary $\p M$.

Denote the solution of (\ref{eq:wave_isotropic}) by $u^f(t,x) = u(t,x)$, let $T > 0$, and define the operator 
\begin{equation}
\label{eq:dtn}
\Lambda_{2T} : f \mapsto u^f|_{(0,2T) \times \p M}.
\end{equation}
Operator $\Lambda_{2T}$ models boundary measurements and is  called the Neumann-to-Dirichlet operator. 
Let us assume that $c|_{\p M}$ is known but $c|_M$ is unknown.
The inverse problem for the wave equation is to reconstruct the wave speed $c(x)$, $x \in M$,
using the operator $\Lambda_{2T}$.

Let $\Gamma \subset \p M$ be open, $\tau \in C(\bar \Gamma)$, and
consider a wave source 
$f$ in $L^2((0,\infty) \times \p M)$ 
satisfying the support condition
\begin{equation}
\label{eq:source_supp_condition}
\supp(f) \subset \{ (t, y) \in [0, T] \times \bar \Gamma;\ t \in [T - \tau(y), T] \}.
\end{equation}
By the finite speed of progation for the wave equation \cite{Ga, Mi},
the solution $u^f$ satisfies then the support condition, 
\begin{equation}
\label{eq:sol_supp_condition}
\supp(u^f(T)) \subset \{x \in M;\ \text{there is $y \in \bar \Gamma$ such that $d(x,y) \le \tau(y)$}\},
\end{equation}
where $d(x,y)$ is the travel time between points $x$ and $y$, see (\ref{domain_of_influence}) below. Let us denote the set in (\ref{eq:sol_supp_condition}) by $M(\Gamma, \tau)$ and call it the domain of influence. 

The contribution of this paper is twofold.
First, we present a method to compute the volume of  $M(\Gamma, \min(\tau, T))$ using the operator $\Lambda_{2T}$.
The method works even when the wave speed is anisotropic, that is, when the wave speed is given by a Riemannian metric tensor $g(x) = (g_{jk}(x))_{j,k}^n$, $x \in \bar M$. 
In the case of the isotropic wave equation (\ref{eq:wave_isotropic}) we have $g(x) = (c(x)^{-2} \delta_{jk})_{j,k}^n$.

Second, assuming that the Riemannian manifold $(\bar M, g)$ is simple, we show that the volumes of $M(\Gamma, \tau)$
for $\tau \in C(\p M)$ contain enough information to 
determine the metric tensor $g$ up to a change of coordinates in $M$.
We recall the definition of a simple compact manifold below,
see Definition \ref{def:simple}.
In the case of the isotropic wave equation (\ref{eq:wave_isotropic}) we can
determine the wave speed $c$
in the Cartesian coordinates of $M$. 

Our method to compute the volume of 
the domain of influence
is a quadratic minimization scheme in $L^2((0, 2T) \times \p M)$ for the source 
$f$ satisfying the support condition (\ref{eq:source_supp_condition}).
After a finite dimensional discretization,
an approximate minimizer 
can be computed by solving
a positive definite system of linear equations. 
We show that the system can be solved very efficiently if 
we use an iterative method, such as the conjugate gradient method,
and intertwine measurements with computation.
In particular, instead of solving the equation (\ref{eq:wave_isotropic}) computationally in an iteration step, we measure $\Lambda_{2T} f$ 
for two sources $f$:
one is the approximate minimizer given by the previous iteration step and the other is related to the time-reversed version of the approximate minimizer, see (\ref{eq:cg_step_measurements}) below.
We believe that our intertwined algorithm is more robust against noise than an algorithm where noise is propagated by simulation of the wave equation.

Let us consider next the problem to determine the metric tensor $g$ given the volumes of $M(\p M, \tau)$ for all 
$\tau \in C(\p M)$. Our approach exploits the fact that
$C(\p M)$ is a lattice with the natural partial order 
\begin{equation}
\tau \le \sigma 
\quad \text{if and only if} \quad 
\tau(y) \le \sigma(y)\ \text{for all $y \in \p M$}.
\end{equation}
Let us define the greatest lower bound of $\tau$ and $\sigma$ in $C(M)$ as their pointwise minimum and
denote it by $\tau \wedge \sigma$.
We recall that a subset of $C(M)$ is a meet-semilattice if it is closed under the binary operation $\wedge$.

Let us define the {\em boundary distance functions},
\begin{equation}
\label{eq:boundary_distance_functions}
r_x : \p M \to [0, \infty), \quad r_x(y) := d(x,y),
\end{equation}
for $x \in \overline M$.
We show that the volumes of 
$M(\p M, \tau)$, $\tau \in C(\p M)$, determine the meet-semilattice, 
\begin{equation}
\label{eq:semilattice_QM_intro}
\overline{Q(M)} = \bigcup_{x \in \overline M}
\{ \tau \in C(M);\ \tau \le r_x \}.
\end{equation}
Moreover, we show that if $(\bar M, g)$ is simple 
then the boundary distance functions are the maximal elements of $\overline{Q(M)}$.
The set of boundary distance functions
determines the Riemannian manifold $(M,g)$ \cite{Ku_proc, KKL}. Thus the volumes of $M(\p M, \tau)$, $\tau \in C(\p M)$, determine $(M,g)$ if it is simple.


Our results give a new uniqueness proof for the 
inverse problem for the wave equation in the case of
a simple geometry. Belishev and Kurylev have proved the uniqueness even when the geometry is not simple \cite{BeKu}. 
Their proof is based on the boundary control method
\cite{AKKLT, Be3, KK, KKLima, Pestov}, originally developed for the isotropic wave equation \cite{Be}. 
Our uniqueness proof might be the first systematic use of lattice structures in the context of inverse boundary value problems.

In previous literature, $M(\Gamma, \tau)$ has been defined in the case of a constant function $\tau$, see e.g. \cite{KKL} and the references therein.
In this paper we establish some properties of $M(\Gamma, \tau)$ when $\Gamma \subset \p M$ is open and $\tau \in C(\bar \Gamma)$.
In particular, we show that its boundary is of measure zero. 
This important detail seems to be neglected in previous literature also in the case of a constant function $\tau$.

Our method to compute the volume of a domain of influence is related to the iterative time-reversal control method 
by Bingham, Kurylev, Lassas and Siltanen \cite{ITRC}.
Their method produces a certain kind of focused waves, 
and they also prove uniqueness for the 
inverse problem for the wave equation using these waves. Moreover, they give a review of 
methods that use time-reversed measurements
\cite{Bardos, Bardos2, Papa1, CIL, FinkD, FinkMain, Kliba}.  
A modification of the iterative time-reversal control method is presented in \cite{DKL}. 

\section{Main results}

Let $(M, g)$ be a $C^\infty$-smooth, 
compact and connected Riemannian manifold 
of dimension $n \ge 2$ with nonempty boundary $\p M$.
We consider the wave equation 
\begin{align}\label{eq:wave}
&\p_t^2 u(t,x) + a(x,D_x) u(t, x) = 0, \quad (t,x) \in (0,\infty) \times M,
\\\nonumber& u|_{t=0} = 0, \quad \p_t u|_{t=0}=0,  
\\\nonumber& b(x, D_x) u(t,x) = f(t,x), \quad (t,x) \in (0,\infty) \times \p M,
\end{align}
where $a(x, D_x)$ is a weighted Laplace-Beltrami operator and $b(x, D_x)$ is the 
corresponding normal derivative. 
In coordinates, $(g^{jk}(x))_{j,k=1}^n$ denotes the inverse of $g(x)$ and $|g(x)|$ the determinant of $g(x)$.
Then
\begin{align*}
a(x,D_x) u
&:= -\sum_{j,k=1}^n \mu(x)^{-1}|g(x)|^{-\frac 12}\frac {\p}{\p x^j} 
\ll( \mu(x)|g(x)|^{\frac 12}g^{jk}(x)\frac {\p u}{\p x^k} \rr),
\\b(x, D_x) u 
&:= \sum_{j,k=1}^n \mu(x)g^{jk}(x) \nu_k(x) \frac{\p u}{\p x^j},
\end{align*}
where $\mu$ is a $C^\infty$-smooth strictly positive weight function and 
$\nu = (\nu_1, \dots, \nu_n)$ is the exterior co-normal vector of $\p M$
normalized with respect to $g$, that is $\sum_{j,k=1}^m g^{jk}\nu_j\nu_k=1$.
The isotropic wave equation (\ref{eq:wave_isotropic}) is a special case of (\ref{eq:wave})
with $g(x) := (c(x)^{-2} \delta_{jk})_{j,k=1}^n$ and $\mu(x) = c(x)^{n-2}$.

We denote the indicator function of a set $A$ by $1_A$,
that is, $1_A(x) = 1$ if $x \in A$ and $1_A(x) = 0$
otherwise. 
Moreover, we denote
\begin{equation} \label{eq:integration_triangle_L}
L := \{ (t,s) \in \R^2; t + s \le 2 T,\ s > t > 0 \},
\end{equation}
and define the operators
\begin{align*}
&J f(t) := \frac{1}{2} \int_0^{2 T} 1_L(t,s) f(s) ds, \quad
R f(t) := f(2 T - t), \quad
\\&K := J \Lambda_{2T} - R \Lambda_{2T} R J, \quad
I f (t) := 1_{(0,T)}(t) \int_0^t f(s) ds,
\end{align*}
where $\Lambda_{2T}$ is the operator defined by (\ref{eq:dtn}) $u^f$ being the solution of (\ref{eq:wave}).
We denote by $dS_g$ the Riemannian volume measure of the manifold $(\p M, g|_{\p M})$.
Furthermore, we denote by $(\cdot, \cdot)$ and $\norm{\cdot}$ 
the inner product and the norm of $L^2((0, 2T) \times \p M; dt \otimes dS_g)$.
We study the regularized minimization problem
\begin{equation}
\label{eq:minimization_regularized}
\argmin_{f \in S} \ll ((f, K f) - 2(I f, 1) + \alpha \norm{f}^2 \rr),
\end{equation}
where the regularization parameter $\alpha$ is strictly positive and $S$
is a closed subspace of $L^2((0, 2T) \times \p M)$.

Operator $a(x, D_x)$ with the domain $H^2(M) \cap H^1_0(M)$ is self-adjoint on
the space $L^2(M; dV_\mu)$, where $dV_\mu = \mu |g|^{1/2} dx$ in coordinates.
Thus we call $dV_\mu$ the natural measure corresponding to $a(x, D_x)$
and denote it also by $m$.
In \cite{ITRC} it is shown that
\begin{equation}
\label{eq:inner_products}
(u^f(T), u^h(T))_{L^2(M; dV_\mu)} 
= (f, K h).
\end{equation}
This is a reformulation of the Blagovestchenskii identity \cite{Bl}.
In Lemma \ref{lem:cross_term} we show the following  identity
\begin{equation}
\label{eq:inner_product_with_1}
(u^f(T), 1)_{L^2(M; dV_\mu)} = (I f, 1).
\end{equation}
This is well known at least in the isotropic case, see e.g. \cite{Be2}.
The equations (\ref{eq:inner_products}) and (\ref{eq:inner_product_with_1}) imply that
\begin{align}
\label{eq:minimization}
(f, K f) - 2(I f, 1)
&= \norm{u^{f}(T)}_{L^2(M; dV_\mu)}^2 - 2 (u^{f}(T), 1)_{L^2(M; dV_\mu)} 
\\\nonumber&=
\norm{u^{f}(T) - 1}_{L^2(M; dV_\mu)}^2 + C,
\end{align}
where $C = -\norm{1}_{L^2(M; dV_\mu)}^2$ does not depend on the source $f$.
Thus the minimization problem (\ref{eq:minimization_regularized}) is equivalent with the 
minimization problem 
\begin{equation*}
\argmin_{f \in S} \ll( \norm{u^{f}(T) - 1}_{L^2(M; dV_\mu)}^2 + \alpha \norm{f}^2 \rr).
\end{equation*}

For $\Gamma \subset M$ and $\tau : \bar \Gamma \to \R$, we define {\em the domain of influence},
\begin{equation}
\label{domain_of_influence}
M(\Gamma, \tau) := \{x \in M;\ \text{there is $y \in \bar \Gamma$ such that $d(x,y) \le \tau(y)$}\},
\end{equation}
where $d$ is the distance on the Riemannian manifold $(M,g)$.
In Section \ref{sec:regularization} we show the following two theorems.
\begin{theorem}
\label{thm:minimization_on_subspace}
Let $\alpha > 0$ and let $S \subset L^2( (0,2T) \times \p M)$ be a closed subspace.
Denote by $P$ the orthogonal projection
\begin{equation*}
P : L^2( (0,2T) \times \p M) \to S.
\end{equation*}
Then the regularized minimization (\ref{eq:minimization_regularized}) has unique minimizer $f_\alpha \in S$, 
and $f_\alpha$ is the unique $f \in S$ solving
\begin{equation}
\label{eq:normal}
(PKP + \alpha) f = P I^+ 1,
\end{equation}
where $I^+$ is the adjoint of $I$ in $L^2( (0,2T) \times \p M)$.
Moreover, $PKP + \alpha$ is positive definite on $S$.
\end{theorem}
\begin{theorem}
\label{thm:indicator_functions}
Let $\Gamma \subset \p M$ be open, $\tau \in C(\bar \Gamma)$ and define
\begin{equation*}
S = \{f \in L^2((0, 2T) \times \p M);\ \text{$\supp(f)$ satisfies (\ref{eq:source_supp_condition})}\}.
\end{equation*}
Let $f_\alpha$, $\alpha > 0$, 
be the minimizer in Theorem \ref{thm:minimization_on_subspace}.
Then in $L^2(M)$
\begin{equation*}
\lim_{\alpha \to 0} u^{f_\alpha}(T) 
= 1_{M(\Gamma, \tau \wedge T)}.
\end{equation*}
\end{theorem}

We denote $m_\tau := m(M(\p M, \tau))$, for $\tau \in C(\p M)$, and $m_\infty := m(M)$.
Moreover, we define
\begin{align}
  Q(M) &:= \{ \tau \in C(\p M);\ m_\tau < m_\infty \}, \label{eq:meet_semilattice_QM}
\\R(M) &:= \{ r_x \in C(\p M);\ x \in M \}, \nonumber
\end{align}
where $r_x$ is the boundary distance function defined by 
(\ref{eq:boundary_distance_functions}) $d$ being the distance on the Riemannian manifold $(M,g)$.
We denote by $\overline{Q(M)}$ the closure of $Q(M)$ in $C(M)$.
In Section \ref{sec:maximal_elements} we prove the 
equation (\ref{eq:semilattice_QM_intro}) and show
the following theorem.
\begin{theorem}
\label{thm:maximal_elements}
If $(M,g)$ satisfies the condition 
\begin{itemize}
\item[(G)] $x_1, x_2 \in M$ and $r_{x_1} \le r_{x_2}$ imply $x_1 = x_2$,
\end{itemize}
then $R(M)$ is the set of maximal elements of $\overline{Q(M)}$.
\end{theorem}

Let $T \ge \max\{ d(x, y) ;\ x \in M,\ y \in \p M \}$.
Then the set of volumes,
\begin{equation}
\label{eq:volumes_for_uniqueness}
\mathcal V := \{m_\tau;\ \tau \in C(\p M)\ \text{and $0 \le \tau \le T$} \},
\end{equation}
determines the set $Q(M)$. Note that $r_x(y) \le T$, for all $x \in M$ and all $y \in \p M$, and that $m_\infty = \max \mathcal V$.
Moreover, by Theorem \ref{thm:indicator_functions} and equation (\ref{eq:inner_products}) we can compute the volume $m_\tau \in \mathcal V$ as the limit 
\begin{equation}
\label{eq:volumes_via_minimizers}
m_\tau = \lim_{\alpha \to 0} (f_\alpha, K f_\alpha).
\end{equation}
The set $R(M)$ determines the manifold $(M,g)$ up to an isometry \cite{Ku_proc, KKL}.
Hence the volumes (\ref{eq:volumes_for_uniqueness}) contain enough information to 
determine the manifold $(M,g)$ in the class of manifolds satisfying (G).
In section \ref{sec:maximal_elements}, we show that simple manifolds satisfy (G).
\begin{definition}
\label{def:simple}
A compact Riemannian manifold $(M, g)$ with boundary is {\em simple}
if it is simply connected, any geodesic has no conjugate points and
$\p M$ is strictly convex with respect to the metric $g$.
\end{definition}

Let us discuss Theorem \ref{thm:minimization_on_subspace} from the point of view of practical computations.
When the subspace $S$ is finite-dimensional, the positive definite system of linear equations (\ref{eq:normal})
can be solved using the conjugate gradient method.
In each iteration step of the conjugate gradient method we must evaluate one matrix-vector product. 
In our case, the product can be realized by two measurements 
\begin{equation}
\label{eq:cg_step_measurements}
\Lambda_{2 T}f, \quad \Lambda_{2 T} R J f, 
\end{equation}
where $f$ is the approximate solution given by the previous iteration step. 
The remaining computational part of the iteration step consists of a few inexpensive vector-vector operations. 
Thus if we intertwine computation of a conjugate gradient steps with measurements (\ref{eq:cg_step_measurements}), 
the computational cost of our method is very low. 

\section{The open and the closed domain of influence}
\label{sec:domains_of_influence}

Let us recall that the domain of influence $M(\Gamma, \tau)$ 
is defined in (\ref{domain_of_influence}) for $\Gamma \subset M$ and $\tau : \bar \Gamma \to \R$.
We call $M(\Gamma, \tau)$ also the {\em closed} domain of influence and
define the {\em open} domain of influence
\begin{equation*}
M^0(\Gamma, \tau) := \{ x \in M;\ \text{there is $y \in \Gamma$ s.t. $d(x,y) < \tau(y)$} \}.
\end{equation*}

Let us consider the closed domain of influence $M(\Gamma, \tau)$ when $\Gamma \subset \p M$ is open and $\tau$ is a constant.
Finite speed of propagation for the wave equation guarantees that the solution $u^f$ at time $T$
is supported on $M(\Gamma, \tau)$ whenever the source $f$ satisfies the support condition 
(\ref{eq:source_supp_condition}).
Moreover, using Tataru's unique continuation result \cite{Ta1, Ta2}, it is possible to show that the set of functions,
\begin{equation*}
\{ u^f(T);\ \text{$f \in L^2((0,2T) \times \p M)$ and $\supp(f)$ satisfies (\ref{eq:source_supp_condition})}\},
\end{equation*}
is dense in $L^2(M^0(\Gamma, \tau))$, see e.g. the proof of Theorem 3.16 and the orthogonality argument of Theorem 3.10 in \cite{KKL}.
It is easy generalize this for $\tau$ of form
\begin{equation*}
\tau(y) = \sum_{j=1}^N T_j 1_{\Gamma_j}(y), \quad y \in \p M,
\end{equation*}
where $N \in \N$, $T_j \in \R$ and $\Gamma_j \subset \p M$ are open, see \cite{ITRC}.
However, the fact that $M(\Gamma, \tau) \setminus M^0(\Gamma, \tau)$ is of measure zero, seems to go unproven in the literature.
In this section we show that this is indeed the case even for $\tau \in C(\bar \Gamma)$.

To our knowledge, this can not be proven just by considering the boundaries of the balls $B(y, \tau(y))$, 
$y \in \Gamma$.
In fact, we give below an example showing that the union of the boundaries $\p B(y, \tau(y))$, for $y \in \p \Gamma$,
can have positive measure.
\begin{example}
Let $\mathcal C$ be the fat Cantor set, $M \subset \R^2$ be open, $g$ be the Euclidean metric, 
$(0,1) \times \{0\} \subset \p M$ and $\Gamma = \ll( (0,1) \setminus \mathcal C \rr) \times \{0\}$.
Then the union $B := \bigcup_{y \in \p \Gamma} \p B(y, 1)$ has positive measure. 
\end{example}
\begin{proof}
The fat Cantor set $\mathcal C$ is an example of a closed subset of $[0,1]$ whose boundary has positive measure, see e.g. \cite{Pugh}.
The map
\begin{equation*}
\Phi : (s, \alpha) \mapsto (s + \cos \alpha, s + \sin \alpha)
\end{equation*}
is a diffeomorphism from $\R \times (0, \pi/2)$ onto its image in $\R^2$.
The image of $H := \p \mathcal C \times (0, \pi/2)$ under $\Phi$
lies in $B$.
As $H$ has positive measure so has $B$.
\end{proof}

\begin{lemma}
\label{lem:characterization_of_domi}
Let $\Gamma \subset \p M$ be open and let $\tau \in C(\bar \Gamma)$. 
Then the function 
\begin{equation*}
r_{\Gamma, \tau}(x) := \inf_{y \in \Gamma} (d(x, y) - \tau(y))
\end{equation*}
is Lipschitz continuous and 
\begin{align}
M(\Gamma, \tau) &= \{ x \in M;\ r_{\Gamma, \tau}(x) \le 0 \}, \label{eq:closed_domi_r}
\\M^0(\Gamma, \tau) &= \{ x \in M;\ r_{\Gamma, \tau}(x) < 0 \}. \label{eq:open_domi_r}
\end{align}
In particular, $M(\Gamma, \tau)$ is closed and $M^0(\Gamma, \tau)$ is open.
\end{lemma}
\begin{proof}
Let us define
\begin{equation*}
r(x) := r_{\Gamma, \tau}(x), \quad \tilde r(x) := \min_{y \in \bar \Gamma} (d(x, y) - \tau(y)),
\end{equation*}
and show that $\tilde r = r$. Clearly $\tilde r\le r$.
Let $x \in M$. The minimum in the definition of $\tilde r(x)$ is attained at a point $y_0 \in \bar \Gamma$.
We may choose a sequence $(y_j)_{j=1}^\infty \subset \Gamma$ such that $y_j \to y_0$ as $j \to \infty$.
Then
\begin{equation*}
r(x) \le d(x, y_j) - \tau(y_j) \to \tilde r(x)
\quad \text{as $j \to \infty$}.
\end{equation*}
Hence $\tilde r = r$.

Let us show that $\tilde r$ is Lipschitz. 
Let $x \in M$, and let $y_0$ be as before. 
Let $x' \in M$. Then 
\begin{equation*}
\tilde r(x') - \tilde r(x) 
\le d(x', y_0) - \tau(y_0) - \ll( d(x, y_0) - \tau(y_0) \rr) \le d(x, x').
\end{equation*}
By symmetry with respect to $x'$ and $x$, $\tilde r$ is Lipschitz.

Let us show (\ref{eq:closed_domi_r}). 
Clearly $r(x) \le 0$ for $x \in M(\Gamma, \tau)$.
Let $x \in M$ satisfy $r(x) \le 0$, and let $y_0$ be as before.
Then 
\begin{equation*}
d(x, y_0) - \tau(y_0) = \tilde r(x) = r(x) \le 0,
\end{equation*}
and $x \in M(\Gamma, \tau)$. Hence (\ref{eq:closed_domi_r}) holds. 
The equation (\ref{eq:open_domi_r}) can be proven in a similar way.
\end{proof}

If $\Gamma \subset \p M$ and $\tau$ is a constant function, then
\begin{align*}
M(\Gamma, \tau) &= \{ x \in M;\ r_{\Gamma, \tau}(x) \le 0 \} 
= \{ x \in M;\ \inf_{y \in \Gamma} d(x,y) \le \tau \}
\\&= \{ x \in M;\ d(x,\Gamma) \le \tau \}.
\end{align*}
Thus for a constant $\tau$, our definition of $M(\Gamma, \tau)$ coincides with the definition 
of the domain of influence in \cite{KKL}. 

\begin{lemma}
\label{lem:level_set_is_null}
Let $A \subset M$ be compact and let $\tau : A \to \R$ be continuous. 
We define
\begin{equation*}
r(x) := \inf_{y \in A} (d(x,y) - \tau(y)), \quad x \in M.
\end{equation*}
If $\tau$ is strictly positive on $A$ or $A$ is a null set, then $\{x \in M;\ r(x) = 0\}$
is a null set.
We mean by a null set a set of measure zero with respect to the Riemannian volume measure.
\end{lemma}
\begin{proof}
Denote by $V_g$ the volume measure of $M$ and define
\begin{equation*}
Z := \{p \in M;\ r(p) = 0\}.
\end{equation*}
Let us show that 
\begin{equation} \label{eq:A_does_not_matter}
V_g(Z) = V_g(Z \setminus A).
\end{equation}
If $V_g(A) = 0$, then (\ref{eq:A_does_not_matter}) is immediate. 
If $\tau > 0$, then 
\begin{equation*}
r(q) \le d(q,q) - \tau(q) = -\tau(q) < 0, \quad q \in A.
\end{equation*}
Hence $Z \cap A = \emptyset$ and (\ref{eq:A_does_not_matter}) holds.

Let $p \in M^\text{int}$. There is a chart $(U, \phi)$ of $M^\text{int}$ such that
$\phi(p) = 0$ and that the closure of the open Euclidean unit ball $B$ of $\R^n$
is contained in $\phi(U)$.
We denote $U_p := \phi^{-1}(B)$.

The sets $U_p$, $p \in M^\text{int}$, form an open cover for $M^\text{int}$,
and as $M^\text{int}$ is second countable, 
there is a countable cover $U_{p_j}$, $j = 1, 2, \dots$, of $M^\text{int}$.
Hence
\begin{equation*}
V_g(Z) = V_g( (Z \setminus A) \cap (\p M \cup \bigcup_{j=1}^\infty U_{p_j}))
\le \sum_{j=1}^\infty V_g( (Z \setminus A) \cap U_{p_j}).
\end{equation*}
It is enough to show that $\phi((Z \setminus A) \cap U_p)$
is a null set with respect to the Lebesgue measure on $B$.

We define for $v = (v^1, \dots, v^n) \in \R^n$ and $x \in B$,
\begin{equation*}
|v|_{g(x)}^2 := \sum_{j,k = 1}^n v^j g_{jk}(x) v^k, 
\quad |v|^2 := \sum_{j = 1}^n (v^j)^2,
\end{equation*}
where $(g_{jk})_{j,k=1}^n$ is the metric $g$ in the local coordinates on $\phi(U)$.
As $\overline B$ is compact in $\phi(U)$, there is $c_p > 0$ such that
for all $v \in \R^n$ and $x \in B$
\begin{equation*}
c_p |v|_{g(x)} \le |v| \le \frac{1}{c_p} |v|_{g(x)}.
\end{equation*}

As in the proof of Lemma \ref{lem:characterization_of_domi}
we see that $r$ is Lipschitz continuous on $M$.
Thus by Rademacher's theorem there is a null set $N \subset B$ such that 
$r$ is differentiable in the local coordinates in $B \setminus N$.
We denote $Z_p := \phi((Z \setminus A) \cap U_p) \setminus N$.

Let $x \in Z_p$ and denote $p_x := \phi^{-1}(x)$.
As $A$ is compact and $q \mapsto d(p_x, q) - \tau(q)$ 
is continuous, there is $q_x \in A$ such that
\begin{equation*}
d(p_x, q_x) - \tau(q_x) = r(p_x) = 0.
\end{equation*}
We denote $s := d(p_x, q_x) = \tau(q_x)$.
As $p_x \notin A$ and $q_x \in A$, we have that $0 < s$.

As $M$ is connected and complete as a metric space, Hopf-Rinow theorem gives
a shortest path $\gamma : [0, s] \to M$ 
parametrized by arclength and joining $q_x = \gamma(0)$ and $p_x = \gamma(s)$.
For a study of shortest paths on Riemannian manifolds with boundary
see \cite{Ax2}.
As $p_x \in U_p$, there is $a \in (0, s)$ such that $\gamma|_{[a , s]}$
is a unit speed geodesic of $U_p \subset M^\text{int}$.
As $\gamma$ is parametrized by arclength,
\begin{equation*}
r(\gamma(t)) \le d(\gamma(t), q_x) - \tau(q_x) = t - s, \quad t \in [a, s],
\end{equation*}
and as $r(\gamma(s)) = r(p_x) = 0$,
\begin{equation*}
\frac{r(\gamma(t)) - r(\gamma(s))}{t - s} \ge \frac{t - s - 0}{t - s} = 1, \quad t \in (a,s).
\end{equation*}

The function $r \circ \gamma$ is differentiable at $s$ by the chain rule, and
\begin{equation*}
\p_t (r \circ \gamma)(s)
= \lim_{t \to s^-} \frac{r(\gamma(t)) - r(\gamma(s))}{t - s} \ge 1.
\end{equation*}
As $\gamma$ is a unit speed geodesic near $s$,
\begin{equation*}
c_p = c_p |\p_t \gamma(s)|_{g(x)} \le |\p_t \gamma(s)| 
\le \frac{1}{c_p}|\p_t \gamma(s)|_{g(x)} = \frac{1}{c_p}.
\end{equation*}
Hence in the local coordinates in $B$
\begin{equation*}
|D r(x)| \ge D r(x) \cdot \frac{\p_t \gamma(s)}{|\p_t \gamma(s)|}
\ge c_p \p_t (r \circ \gamma)(s) \ge c_p.
\end{equation*}

Let $\epsilon > 0$. 
There is $\delta(x) > 0$ such that 
\begin{equation}
\label{eq:r_derivative_approximation}
|r(y) - r(x) - D r(x) \cdot (y - x)| \le c_p \epsilon |y - x|, \quad y \in B(x, \delta(x)),
\end{equation}
and $B(x, \delta(x)) \subset B$.
Here $B(x, \delta)$ is the open Euclidean ball with center $x$ and radius $\delta$.

The sets $B(x, \delta(x)/5)$, $x \in Z_p$, form an open cover for $Z_p$,
and as $Z_p$ is second countable, 
there is $(x_j)_{j=1}^\infty \subset Z_p$ such that the sets 
\begin{equation*}
B_j' := B(x_j, \delta(x_j)/5)
\end{equation*}
form an open cover for $Z_p$.
By Vitali covering lemma there is an index set $J \subset \N$ such that
the sets $B_j'$, $j \in J$, are disjoint and the sets
\begin{equation*}
B_j := B(x_j, \delta(x_j)), \quad j \in J,
\end{equation*}
form an open cover for $Z_p$.

We denote 
\begin{equation*}
v_j := \frac{D r(x_j)}{|D r(x_j)|}, \quad \delta_j := \delta(x_j).
\end{equation*}
If $y \in Z_p \cap B_j$, then $r(y) = 0 = r(x_j)$ and by (\ref{eq:r_derivative_approximation})
\begin{equation*}
|v_j \cdot (y - x_j)| \le \frac{1}{|D r(x_j)|} c_p \epsilon |y - x_j| \le \epsilon \delta_j.
\end{equation*}

We denote by $\alpha_m$ the volume of the open Euclidean unit ball in $\R^m$
and by $V$ the Lebesgue measure on $B$.
Let $j \in J$.
Using a translation and a rotation we get such 
coordinates that $x_j = 0$ and $v_j = (1, 0, \dots, 0)$.
In these coordinates 
\begin{equation*}
Z_p \cap B_j \subset \{ (y^1, y') \in \R \times \R^{n-1};\ |y^1| \le \epsilon \delta_j, |y'| \le \delta_j \}.
\end{equation*}
Hence $V(Z_p \cap B_j) \le 2 \epsilon \delta_j \alpha_{n-1} \delta_j^{n-1}$.
Particularly,
\begin{equation*}
V(Z_p \cap B_j) 
\le \epsilon \frac{2 \alpha_{n-1}}{\alpha_{n}} V(B_j) 
= \epsilon c_n V(B_j'),
\end{equation*}
where $c_n := 2 \cdot 5^n \alpha_{n-1} / \alpha_{n}$.
Then
\begin{align*}
V(\phi((Z \setminus A) \cap U_p))
&= V(Z_p) = V(Z_p \cap \bigcup_{j \in J} B_j) 
\le \sum_{j \in J} V(Z_p \cap B_j)
\\&\le \epsilon c_n \sum_{j \in J} V(B_j') 
= \epsilon c_n V(\bigcup_{j \in J} B_j') 
\le \epsilon c_n V(B).
\end{align*}
As $\epsilon > 0$ is arbitrary, $V(\phi((Z \setminus A) \cap U_p)) = 0$ and the claim is proved.
\end{proof}

\section{Approximately constant wave fields on a domain of influence}
\label{sec:regularization}

\begin{lemma}
\label{lem:cross_term}
Let $f \in L^2((0, 2T) \times \p M)$. Then the equation (\ref{eq:inner_product_with_1}) holds.
\end{lemma}
\begin{proof}
The map $h \mapsto u^h(T)$ is bounded $L^2((0, 2T) \times \p M) \to L^2(M)$, see e.g. \cite{LaTr}.
Thus is it enough to prove the equation (\ref{eq:inner_product_with_1}) for $f \in C_c^\infty( (0, 2T) \times \p M)$.
Let us denote
\begin{equation*}
v(t) := (u^f(t), 1)_{L^2(M; dV_\mu)}.
\end{equation*}
As $a(x,D_x) 1 = 0$ and $b(x,D_x)1 = 0$,
we may integrate by parts 
\begin{align*}
\p_t^2 v(t) &= - (a(x, D_x) u^f(t), 1)_{L^2(M; dV_\mu)}
\\&= - \ll( (a(x, D_x) u^f(t), 1)_{L^2(M; dV_\mu)} - (u^f(t), a(x, D_x) 1)_{L^2(M; dV_\mu)} \rr)
\\&= (b(x, D_x) u^f(t), 1)_{L^2(\p M; dS_g)} - (u^f(t), b(x, D_x) 1)_{L^2(\p M; dS_g)}
\\&= (f(t), 1)_{L^2(\p M; dS_g)}.
\end{align*}
As $\p_t^j v(0) = 0$ for $j=0,1$,
\begin{align*}
v(T) &= \int_0^T \int_0^t \int_{\p M} f(s,x) dS_g(x) ds dt
\\&= \int_0^{2T} \int_{\p M} 1_{(0,T)}(t) \int_0^t f(s,x) ds dS_g(x) dt.
\end{align*}
\end{proof}

The proof of Theorem \ref{thm:minimization_on_subspace} 
is similar to the proof of the corresponding result in \cite{ITRC}.
We give the proof for the sake of completeness.

\begin{proof}[Proof of Theorem \ref{thm:minimization_on_subspace}.]
We define
\begin{equation*}
E(f) := (f, Kf) - 2 (If, 1) + \alpha \norm{f}^2.
\end{equation*}
Then 
\begin{equation}
\label{eq:energy_functional}
E(f) = \norm{u^{f}(T) - 1}_{L^2(M; dV_\mu)}^2 - \norm{1}_{L^2(M; dV_\mu)}^2 + \alpha \norm{f}^2.
\end{equation}
Let $(f_j)_{j=1}^\infty \subset S$ be such that 
\begin{equation*}
\lim_{j \to \infty} E(f_j) = \inf_{f \in S} E(f).
\end{equation*}
Then 
\begin{equation*}
\alpha \norm{f_j} \le E(f_j) + \norm{1}_{L^2(M; dV_\mu)}^2,
\end{equation*}
and $(f_j)_{j=1}^\infty$ is bounded in $S$.
As $S$ is a Hilbert space,
there is a subsequence of $(f_j)_{j=1}^\infty$ converging weakly in $S$.
Let us denote the limit by $f_\infty \in S$ and the subsequence still by $(f_j)_{j=1}^\infty$.

The map $h \mapsto u^h(T)$ is bounded
\begin{equation*}
L( (0, 2T) \times \p M) \to H^{5/6 - \epsilon}(M)
\end{equation*}
for $\epsilon > 0$, see \cite{LaTr}.
Hence $h \mapsto u^h(T)$ is a compact operator
\begin{equation*}
L^2( (0, 2T) \times \p M) \to L^2(M),
\end{equation*}
and $u^{f_j}(T) \to u^{f_\infty}(T)$ in $L^2(M)$ as $j \to \infty$.
Moreover, the weak convergence implies 
\begin{equation*}
\norm{f_\infty} \le \liminf_{j \to \infty} \norm{f_j}.
\end{equation*}
Hence
\begin{align*}
E(f_\infty) &= \lim_{j \to \infty} \norm{u^{f_j}(T) - 1}_{L^2(M; dV_\mu)}^2 - \norm{1}_{L^2(M; dV_\mu)}^2 + \alpha \norm{f_\infty}^2
\\&\le \lim_{j \to \infty} \norm{u^{f_j}(T) - 1}_{L^2(M; dV_\mu)}^2 - \norm{1}_{L^2(M; dV_\mu)}^2 + \alpha 
\liminf_{j \to \infty} \norm{f_j}^2
\\&= \liminf_{j \to \infty} E(f_j) = \inf_{f \in S} E(f),
\end{align*}
and $f_\infty \in S$ is a minimizer.

Let $f_\alpha$ be a minimizer and $h \in S$. 
By orthogonality of the projection $P$ and identity (\ref{eq:inner_products}),
it is clear that $PKP$ is self-adjoint and positive semidefinite.
Denote by $D_h$ the Fr\'echet derivative to direction $h$. 
As $f_\alpha = Pf_\alpha$ and $h = Ph$
\begin{equation*}
0 = D_h E(f_\alpha) = 2 (h, PKP f_\alpha) - 2 (h, PI^+ 1) + 2 \alpha (h, f_\alpha).
\end{equation*}
Hence $f_\alpha$ satisfies (\ref{eq:normal}).
As $PKP$ is positive semidefinite, $PKP + \alpha$ is positive definite
and solution of (\ref{eq:normal}) is unique.
\end{proof}

\begin{lemma}
\label{lem:approximation_by_simple}
Let $\Gamma \subset \p M$ be open, $\tau \in C(\bar \Gamma)$ and $\epsilon > 0$.
Then there is a simple function
\begin{equation*}
\tau_\epsilon(y) = \sum_{j=1}^N T_j 1_{\Gamma_j}(y), \quad y \in \p M,
\end{equation*}
where $N \in \N$, $T_j \in \R$ and $\Gamma_j \subset \Gamma$ are open,
such that 
\begin{align*}
\tau - \epsilon &< \tau_\epsilon \quad \text{almost everywhere on $\Gamma$ and}
\\\tau_\epsilon &< \tau \quad \text{on $\bar \Gamma$}.
\end{align*}
\end{lemma}
\begin{proof}
As $\p M$ is compact, there is a finite set of coordinate charts covering $\p M$.
Using partition of unity, we see that it is enough to prove the claim in the case 
when $\Gamma \subset \R^{n-1}$ is an open set. 
But then $\tau$ is a continuous function on a compact set $\overline \Gamma \subset \R^{n-1}$,
and it is clear that there is a simple function with the required properties. 
%
\end{proof}

\begin{lemma}
\label{lem:measures_by_approximation}
Let $\Gamma \subset \p M$ be open, $\tau \in C(\bar \Gamma)$ and let $\tau_\epsilon$, $\epsilon > 0$, satisfy
\begin{align*}
\tau - \epsilon &< \tau_\epsilon \quad \text{almost everywhere on $\Gamma$ and}
\\\tau_\epsilon &< \tau \quad \text{on $\bar \Gamma$}.
\end{align*}
Then
\begin{equation*}
\lim_{\epsilon \to 0} m(M(\Gamma, \tau_\epsilon)) = m(M(\Gamma, \tau)).
\end{equation*}
\end{lemma}
\begin{proof}
Let $\epsilon > 0$ and denote by $N \subset \Gamma$ 
the set of measure zero where $\tau - \epsilon \ge \tau_\epsilon$ as functions on $\Gamma$.
Let us show that $M^0(\Gamma, \tau - \epsilon) \subset M(\Gamma, \tau_\epsilon)$.
Let $x \in M^0(\Gamma, \tau - \epsilon)$. Then there is $y_0 \in \Gamma$
such that 
\begin{equation*}
d(x,y_0) < \tau(y_0) - \epsilon.
\end{equation*}
As $\tau$ and the function $y \mapsto d(x,y)$ are continuous
and $\Gamma \setminus N$ is dense in $\Gamma$, there is $y \in \Gamma \setminus N$ such that
\begin{equation*}
d(x,y) < \tau(y) - \epsilon < \tau_\epsilon(y).
\end{equation*}
Hence $x \in M(\Gamma, \tau_\epsilon)$.
A similar argument shows that $M(\Gamma, \tau_\epsilon) \subset M^0(\Gamma, \tau)$.

Clearly $M^0(\Gamma, \tau - \epsilon_1) \subset M^0(\Gamma, \tau - \epsilon_2)$ for $\epsilon_1 \ge \epsilon_2 > 0$,
and
\begin{equation*}
\bigcup_{\epsilon > 0} M^0(\Gamma, \tau - \epsilon) = M^0(\Gamma, \tau).
\end{equation*}
Hence $m(M^0(\Gamma, \tau - \epsilon)) \to m(M^0(\Gamma, \tau))$ as $\epsilon \to 0$, and
\begin{align*}
0 \le m(M^0(\Gamma, \tau)) - m(M(\Gamma, \tau_\epsilon)) &\le m(M^0(\Gamma, \tau)) - m(M^0(\Gamma, \tau - \epsilon))
\\&\to 0, \quad \text{as $\epsilon \to 0$}.
\end{align*}
Moreover, by Lemmas \ref{lem:characterization_of_domi} and \ref{lem:level_set_is_null}
\begin{equation*}
m(M(\Gamma, \tau)) = m(M^0(\Gamma, \tau)) = \lim_{\epsilon \to 0} m(M(\Gamma, \tau_\epsilon)).
\end{equation*}
\end{proof}

\begin{proof}[Proof of Theorem \ref{thm:indicator_functions}.]
We may assume without loss of generality that $\tau \le T$, as
we may replace $\tau$ by $\tau \wedge T$ in what follows.
Let us denote
\begin{equation*}
S(\Gamma, \tau) := \{ f \in L^2((0,2T) \times \p M);\ \text{$\supp(f)$ satisfies (\ref{eq:source_supp_condition})}\}.
\end{equation*}
By the finite speed of propagation for the wave equation,
we have that $\supp(u^f(T)) \subset M(\Gamma, \tau)$ whenever $f \in S(\Gamma, \tau)$.
Hence for $f \in S(\Gamma, \tau)$,
\begin{align}
\label{eq:splitting_the_difference_1}
\norm{u^f(T) - 1}_{L^2(M; dV_\mu)}^2
&= \int_{M(\Gamma, \tau)} (u^f(T) - 1)^2 dV_\mu + \int_{M \setminus M(\Gamma, \tau)} 1 dV_\mu
\\\nonumber&= \norm{u^f(T) - 1_{M(\Gamma, \tau)}}_{L^2(M; dV_\mu)}^2 
    + m(M \setminus M(\Gamma, \tau)).
\end{align}

Let $\epsilon > 0$.
By Lemmas \ref{lem:approximation_by_simple} and \ref{lem:measures_by_approximation}
there is a simple function $\tau_\delta$ satisfying
\begin{equation*}
\tau_\delta < \tau, \quad m(M(\Gamma, \tau)) - m(M(\Gamma, \tau_\delta)) < \epsilon.
\end{equation*}
By the discussion in the beginning of Section \ref{sec:domains_of_influence}, the set 
\begin{equation*}
\{ u^f(T) \in L^2(M(\Gamma, \tau_\delta));\ f \in S(\Gamma, \tau_\delta) \}
\end{equation*}
is dense in $L^2(M(\Gamma, \tau_\delta))$.
Thus there is $f \in S(\Gamma, \tau_\delta) \subset S(\Gamma, \tau)$ such that
\begin{equation*}
\norm{u^f(T) - 1_{M(\Gamma, \tau_\delta)}}_{L^2(M; dV_\mu)}^2 \le \epsilon.
\end{equation*}
Then 
\begin{equation*}
\norm{u^f(T) - 1_{M(\Gamma, \tau)}}_{L^2(M; dV_\mu)}^2 \le 
\epsilon + \norm{1_{M(\Gamma, \tau_\delta)} - 1_{M(\Gamma, \tau)}}_{L^2(M; dV_\mu)}^2
\le 2 \epsilon.
\end{equation*}

Moreover, $E(f_\alpha) \le E(f)$ and equations (\ref{eq:splitting_the_difference_1}) and (\ref{eq:energy_functional}) give
\begin{align*}
&\norm{u^{f_\alpha}(T) - 1_{M(\Gamma, \tau)}}_{L^2(M; dV_\mu)}^2   
\\&\quad= \norm{u^{f_\alpha}(T) - 1}_{L^2(M; dV_\mu)}^2 - m(M \setminus M(\Gamma, \tau))
\\&\quad\le E(f_\alpha) + \norm{1}_{L^2(M; dV_\mu)}^2 - m(M \setminus M(\Gamma, \tau))
\\&\quad\le \norm{u^f(T) - 1}_{L^2(M; dV_\mu)}^2 - m(M \setminus M(\Gamma, \tau)) + \alpha \norm{f}^2
\\&\quad= \norm{u^f(T) - 1_{M(\Gamma, \tau)}}_{L^2(M; dV_\mu)}^2 + \alpha \norm{f}^2 
\le 2 \epsilon + \alpha \norm{f}^2.
\end{align*}
We may choose first small $\epsilon > 0$ and then small $\alpha > 0$ to 
get $u^{f_\alpha}(T)$ arbitrarily close to $1_{M(\Gamma, \tau)}$ in $L^2(M)$. 
\end{proof}

\section{The boundary distance functions as maximal elements}
\label{sec:maximal_elements}

\def\tilde{\widetilde}
We denote $M^0(\tau) := M^0(\p M, \tau)$, for $\tau \in C(\p M)$, and define
\begin{equation*}
\tilde Q := \{ \tau \in C(M);\ M \setminus M^0(\tau) \ne \emptyset \}.
\end{equation*}

\begin{lemma}
\label{lem:maximal_elements}
If $\tau$ is a maximal element of $\tilde Q$, then $\tau = r_x$ 
for some $x \in M$.
Moreover, if the manifold $(M,g)$ satisfies (G),
then $R(M)$ is the set of the maximal elements of $\tilde Q$.
\end{lemma}
\begin{proof}
Let $x \in M$ and $\tau \in C(\p M)$. Then $\tau \le r_x$ if and only if $x \notin M^0(\tau)$.
In other words,
\begin{equation}
\label{eq:characterization_of_Q}
\tilde Q = \{ \tau \in C(\p M);\ \text{there is $x \in M$ such that $\tau \le r_x$} \}.
\end{equation}
Moreover, $r_x \in \tilde Q$ for all $x \in M$.
Indeed, $r_x$ is continuous and trivially $r_x \le r_x$.

Suppose that $\tau$ is a maximal element of $\tilde Q$. 
By (\ref{eq:characterization_of_Q}) there is $x \in M$ such that $\tau \le r_x$,
but $r_x \in \tilde Q$ and maximality of $\tau$ yields $\tau = r_x$

Let us now suppose that $(M,g)$ satisfies (G) and show that $r_x$ 
is a maximal element of $\tilde Q$.
Suppose that $\tau \in \tilde Q$ satisfies $r_x \le \tau$.
By (\ref{eq:characterization_of_Q}) there is $x' \in M$ such that $\tau \le r_x'$.
Hence 
\begin{equation*}
r_x \le \tau \le r_{x'},
\end{equation*}
and (G) yields that $x = x'$. 
Thus $\tau = r_x$ and  $r_x$ is a maximal element of $\tilde Q$.
\end{proof}

\begin{lemma}
\label{lem:closure}
The set $\tilde Q$ is the closure of $Q(M)$ in $C(M)$.
\end{lemma}
\begin{proof}
Let us first show that $\tilde Q$ is closed.
Let $(\tau_j)_{j=1}^\infty \subset \tilde Q$ satisfy $\tau_j \to \tau$ in $C(\p M)$ as $j \to \infty$.
By (\ref{eq:characterization_of_Q}) there is $(x_j)_{j=1}^\infty \subset M$ such that $\tau_j \le r_{x_j}$.
As $M$ is compact there is a converging subsequence $(x_{j_k})_{k=1}^\infty \subset (x_j)_{j=1}^\infty$.
Let us denote the limit by $x$, that is, $x_{j_k} \to x$ as $k \to \infty$.
By continuity of the distance function,
\begin{equation*}
\tau(y) = \lim_{k \to \infty} \tau_{j_k}(y) \le \lim_{k \to \infty} r_{x_{j_k}}(y) = r_x(y), \quad y \in \p M.
\end{equation*}
Hence $\tau \in \tilde Q$ and $\tilde Q$ is closed.

Clearly $Q(M) \subset \tilde Q$ and it is enough to show that $Q(M)$ is dense in $\tilde Q$.
Suppose that $\tau \in \tilde Q$. Then there is $x_0 \in M$ such that $\tau \le r_{x_0}$.
Let $\epsilon > 0$. As $M \times \p M$ is compact and the distance function is continuous,
there is $r > 0$ such that
\begin{equation*}
\sup_{y \in \p M} |d(x, y) - d(x_0, y)| < \epsilon, \quad \text{when $d(x, x_0) < r$ and $x \in M$}.
\end{equation*}
Hence $\tau(y) - \epsilon \le r_{x_0}(y) - \epsilon < r_x(y)$ for all $y \in \p M$ and all $x \in B(x_0, r)$.
In other words, 
\begin{equation*}
B(x_0, r) \subset (M \setminus M(\tau - \epsilon)),
\end{equation*}
and this yields that $\tau - \epsilon \in Q(M)$. 
Functions $\tau - \epsilon$ converge to $\tau$ in $C(\p M)$ as $\epsilon \to 0$. Thus $\tau$ is in the closure of $A$.
\end{proof}

Lemmas \ref{lem:maximal_elements} and \ref{lem:closure} together prove Theorem \ref{thm:maximal_elements}.
Moreover, (\ref{eq:characterization_of_Q})
yields the equation (\ref{eq:semilattice_QM_intro})
in the introduction.

\begin{lemma}
\label{lem:simple_manifold}
If $(M,g)$ is simple or the closed half sphere, then (G) holds.
\end{lemma}
\begin{proof}
Let $x_1, x_2 \in M$ satisfy $x_1 \ne x_2$, and let us show that $r_{x_1} \nleq r_{x_2}$.
First, if $x_2 \in \p M$ then $r_{x_1}(x_2) > 0 = r_{x_2}(x_2)$.
Second, if $x_2 \in M^{\text{int}}$, then there
is the unique unit speed geodesic $\gamma$ and the unique point $y \in \p M$ such that 
$\gamma(0) = x_1$, $\gamma(s) = x_2$ and $\gamma(s') = y$,
where $0 < s < s'$.
As $\gamma$ is a shortest path
from $x_1$ to $y$ (the shortest path if $(M,g)$ is simple) and the shortest path from $x_2$ to $y$,
\begin{equation*}
r_{x_1}(y) = s' > s = r_{x_2}(y).
\end{equation*}
\end{proof}

The closed half sphere is not simple, since for a point on the boundary the corresponding 
antipodal point is a conjugate point.
Hence the manifolds satisfying (G) form a strictly larger class than the simple manifolds.

\vspace{1cm}
{\em Acknowledgements.}
The author would like to thank Y. Kurylev for useful discussions. 
The research was partly supported by Finnish Centre of Excellence in Inverse Problems Research,
Academy of Finland COE 213476,
and partly by Finnish Graduate School in Computational Sciences.


Electronic mail address of the author:
\url{lauri.oksanen@helsinki.fi}

\end{document}